%% file: Fair.tex
\documentclass{amsart}

\usepackage{amsmath, subfig}%
\usepackage[figuresright]{rotating}
\usepackage{epstopdf}
\input{style}
\input{comments}

\author{Duy Anh Nguyen} 
\thanks{The author is with the Hanoi Obstetrics and Gynecology Hospital, Vietnam.}
\title[Distributionally Robust Fair Workload Balancing]{Workload Balancing Among Heathcare Workers Under Uncertain Service Time Using Distributionally Robust Optimization}

\date{\today}

\begin{document}

\maketitle

\begin{abstract}
Healthcare systems are facing serious challenges in balancing their human resources to cope with volatile service demand, while at the same time providing necessary job satisfaction to the healthcare workers. We propose in this paper a distributionally robust optimization formulation to generate a task assignment plan that promotes the fairness in allocation, attained by reducing the difference in the total working time among workers, under uncertain service time. The proposed joint chance constraint model is conservatively approximated by a worst-case Conditional Value-at-Risk, and we devise a sequential algorithm to solve the finite-dimensional reformulations which are linear (mixed-binary) optimization problems. We also provide explicit formula in the situation where the support set of the random vectors is a hyperrectangle. The experiment with synthetic data suggests promising results for our approach.

\end{abstract}

\section{Introduction}
\label{sect:intro}

The fast aging population and devastating pandemic outbreaks are exerting extreme pressure on the worldwide healthcare systems. Apart from the lack of medical devices, prescriptive drugs and the lack of developed infrastructure, healthcare providers are in critical shortage of qualified workers, including surgeons, medical doctors and nurses~\cite{ref:aluttis2014workforce}. Many surveys indicate that a majority of healthcare workers are not satisfied with their occupation due to low salary, tremendous emotional stress and long shift working hours~\cite{ref:khani2008quality, ref:moradi2014quality, ref:suleiman2019quality}. Because the healthcare workers hold direct responsibility to communicate and provide service to patient, ensuring their job satisfaction is an important factor to delivery the highest possible quality of treatment and care to the community in need~\cite{ref:gabrani2016perceived,ref:seda2019multi, ref:permarupan2020nursing}. Hence, modern hospital management needs to deal with numerous complex and conflicting criteria in the daily operations~\cite{ref:landa2018multi}.

Apart from the provision of necessary support and attractive compensation, a reasonable work schedule is arguably one of the most critical factor to promote job satisfaction in healthcare workers~\cite{ref:glass1010nurse, ref:nelson2018development}. Besides the number of working hours per day/week and the characteristics of the work rotation, the current research is focusing on promoting fairness in the scheduling or the rostering of the workers. While there are many definitions for the notion of fairness in the existing literature~\cite{ref:ibrahim2011nurse, ref:sangai2017workload, ref:uhde2020fairness}, in this paper, we focus on a specific choice of fairness that ensure a balance workload among health workers. Workload balancing has also been studied in the literature related to scheduling, where the focus has been placed on the static setting~\cite{ref:maenhout2010branching, ref:burke2010hybrid, ref:bai2010hybrid}.

In this paper, we use the duration of the task as a proxy to calculate the workload of the tasks, and as such, other factors which may directly contribute to the workload such as the complexity of the task are omitted. Balancing the workload in this case is simplified to generating a shift schedule in which all healthcare workers will spend a relatively similar amount of time to finish the assigned jobs. In reality, the exact duration of each task, also referred to as the service time, is rarely known in advanced, and thus, it is difficult to correctly specify the true workload given to each worker. If we use random vectors to model the service time, then in many situations, we may not have access to the true distribution of these random vectors. At best, we may have access to limited training or historical data, from which we can try to infer the ambiguity set which contains the true distribution with high confidence. Consequentially, workload balancing under uncertainty remains a challenging and active area for both theoretical research as well as operational implementation~\cite{ref:mousazadeh2018health}.

There are several attractive frameworks to deal with the inherent uncertain nature of healthcare decision making at both the operational and the planning level including stochastic programming~\cite{ref:shapiro2014lectures} and robust optimization~\cite{ref:bental09robust}.
Distributionally robust optimization is an emerging framework for decision support when the information about the distribution of the underlying randomness is incorrect or incomplete~\cite{ref:delage2010distributionally, ref:kuhn2019wasserstein}. Distributionally robust solutions have been shown to deliver superior performance in many healthcare decision making tasks, including, but not limited to, appointment scheduling \cite{ref:zhang2017dist, ref:kong2020appointment}, nurse staffing~\cite{ref:ryu2019nurse}, surgery block allocation~\cite{ref:wang2019dist}, and location planning of emergency service~\cite{ref:liu2019dist}.

This paper aims to bridge the existing gap in the literature of fair scheduling of healthcare workers by considering a flexible model to solve the workload balancing under uncertain service time using the distributionally robust optimization framework. This paper stems from one of the long-standing efforts of the management team at our hospital\footnote{Hospital name is anonymized} to improve the scheduling and rostering of our healthcare workers, especially in improving their job satisfaction.  The contributions of this paper are highlighted as follows. 
\begin{itemize}
	\item We detail a joint chance constraint optimization problem that aims to produce a fair allocation of tasks among a team of healthcare workers. This optimization problem requires that the difference of the working time among healthcare workers is bounded by a threshold with high probability. Based on this formulation, we propose the distributionally robust optimization problem with joint chance constraint to deal with situations when the true underlying probability of the randomness is difficult to be identified, or when this probability may entail the risk of misspecification. 
	\item We develop a conservative approximation of the joint chance constraint using the worst-case Conditional Value-at-Risk when the ambiguity set contains the expected values and support information of the true distribution. Furthermore, we propose a sequential optimization algorithm to efficiently solve the resulting reformulations, which are often linear (binary) optimization problems. To facilitate practical implementation of the approach, we also tailor the result to the special case where the random vector is supported on a hyperrectangular set.
\end{itemize}

This paper unfolds as follows. Section~\ref{sect:formulations} layouts our mathematical framework for workload balancing under uncertain service time. Section~\ref{sect:approx} details the approximation scheme, which will be tailored to the hyperrectangular support set in Section~\ref{sect:hypercube}. Section~\ref{sect:numerical} reports the numerical results and Section~\ref{sect:conclusions} concludes the paper with future research directions. 

\textbf{Notations.} At any given risk level $\eps \in (0, 1)$, the Conditional Value-at-Risk of a measurable loss function $\ell(\xi)$ dependent on the random vector $\xi$ whose distribution is governed by a probability measure $\QQ$ is defined as $\QQ\text{-CVaR}_\eps(\ell(\xi)) \Let \inf_{\tau} \left\{ \tau + \eps^{-1} \EE_\QQ[ \max\{0, \ell(\xi) - \tau\}] \right\}$, where the operator $\EE_\QQ$ denotes the integration under the probability measure $\QQ$. We use $|\mc A|$ to denote the cardinality of the set $\mc A$. For any matrix (vector) $A$, we denote by $A^\top$ its transpose. For two vectors $l$ and $u$ of the same dimension, the inequalities $l \le u$ and $l < u$ are understood in the element-wise sense. For two integers $a \leq b$, we use $\llbracket a, b\rrbracket$ to denote the set $\{a, a+1, \ldots, b\}$.
\section{Model Formulations}
\label{sect:formulations}

We consider a healthcare system with a collection $\mc J$ of healthcare workers. Given a set of tasks $\mc I$, the planner aims to find the best plan to assign each task $i \in \mc I$ to a worker $j \in \mc J$ to maximize the system output, while at the same time ensures that the total workload assigned to each worker $j \in \mc J$ is comparatively balanced. We assume that task $i \in \mc I$ utilizes a random amount of resource $\xi_i$, and assigning task $i \in \mc I$ to worker $j \in \mc J$ brings a deterministic reward $r_{ij} \in \R$. If we use the binary decision variables $x_{ij}$ with
\[
x_{ij} = \begin{cases}
1 & \text{if task $i$ is assigned to worker $j$,} \\
0 & \text{otherwise.}
\end{cases}
\]
to carry the information of the task assignment to each worker, then the total time that takes worker $j$ to finish all given tasks can be written as the sum $\sum_{i \in \mc I} \xi_i x_{ij}$. To promote the fairness in the allocation of tasks to workers, it is plausible to impose that all workers will require exactly the same total amount of time to finish the given tasks, which can be equivalently translated to a mathematical constraint of the form
\[
	\sum_{i \in \mc I} \xi_i x_{ij} = \sum_{i \in \mc I} \xi_i x_{ij'}  \quad \forall (j, j') \in \mc J \times \mc J.
\]
Unfortunately, because $\xi$ are random service time, the above constraint can hardly be met in many settings. Indeed, if $\xi$ is governed by a probability measure which is absolutely continuous with respect to a Lebesgue measure then this constraint is met with probability 0 irrespective of any choice of assignment $x$. A possible remedy is to constraint the time difference to be lower than a positive threshold $\Delta \in \R_+$ as
\[
	\Big| \sum_{i \in \mc I} \xi_i (x_{ij} - x_{ij'} ) \Big| \leq \Delta \quad \forall (j, j') \in \mc J \times \mc J,
\]
where the inequality should be understood as an almost-surely inequality, which basically means it should hold with probability one. Unfortunately, this almost-surely constraint is extremely pessimistic and oftentimes it leads to the infeasibility of the optimization problem. Instead of imposing this constraint, we seek to propose a relaxation using the joint chance constraint reformulation. Thus, we propose the \textit{stochastic fairness-constrained allocation problem} which can be written as a joint chance constrained optimization problem
\be \label{eq:stoprog}
    \begin{array}{cl}
        \max & f(x) \Let \ds \sum_{ (i, j) \in \mc I \times \mc J} r_{ij} x_{ij}  \\
        \st & x \in \mathbb X \\
        & \ds \PP\left( \Big| \sum_{i \in \mc I} \xi_i (x_{ij} - x_{ij'}) \Big| \le \Delta \; \forall (j, j') \right) \ge 1-\eps.
    \end{array} 
\ee
In problem~\eqref{eq:stoprog}, the set of feasible assignment $\mbb X$ is defined as
\[
    \mbb X \Let \left\{
        x \in \mbb \{0, 1\}^{| \mc I | \times | \mc J|}: \sum_{j} x_{ij} = 1 \; \forall i \in \mc I
    \right\},
\]
where the constraints defining $\mbb X$ indicate that any task $i$ is assigned to exactly one worker. The objective function $f(x)$ quantifies the total reward of each allocation $x$ and it is assumed to be an affine function of $x$. The joint chance constraint in~\eqref{eq:stoprog} is prescribed using a balance threshold $\Delta \in \R_+$ and a small tolerance $\eps \in (0, 1)$. The sign $\forall (j, j')$ in this constraint is understood as for all pairs of workers $(j, j') \in \mc J \times \mc J$. This joint chance constraint dictates that the comparative workload among all workers is balanced up to a threshold $\Delta$ with high probability $1 - \eps$. 

Problem~\eqref{eq:stoprog} is a stochastic program which requires the true probability distribution $\PP$ of the (joint) distribution of the random resource utilization $\xi$ as input. Unfortunately, evaluating whether an allocation $x$ is feasible for the joint chance constraint of~\eqref{eq:stoprog} is in general \#P-hard~\cite{ref:dyer1988on}. Moreover, due to the complex nature of the tasks in the healthcare system, full knowledge about $\PP$ is rarely known in practice. At best, the healthcare system may have access to some historical data which can be used to infer $\PP$. However, the available data may be very scarce and may not represent well the distribution of certain task. 

To alleviate the downsides of problem~\eqref{eq:stoprog}, we propose to approach problem~\eqref{eq:stoprog} using the lens of distributionally robust optimization. To this end, we define the ambiguity set
\be \label{eq:ambi-def}
    \mc Q \Let\left\{ \QQ \in \mc P(\mbb R^{|\mc I|}): \EE_{\QQ}[\xi] = \mu,~\QQ(\xi \in \Xi) = 1 \right\},
\ee
where $\mc P(\R^{|\mc I|})$ is the set of all probability measures on $\R^{|\mc I|}$ and the support set $\Xi$ is a compact polyhedron (polytope) that can be described effectively by $M$ linear constraints of the form
\be \label{eq:Xi-def}
    \Xi = \left\{
        \xi \in \mbb R^{|\mc I|}: G \xi \leq h
    \right\}
\ee
for some matrix $G \in \R^{M \times |\mc I|}$ and vector $h \in \R^M$. Furthermore, we assume that the mean vector $\mu$ belongs to the interior of the support $\Xi$, that means $\m \in \mathrm{int}(\Xi)$, or equivalently $G \mu < h$. Descriptively speaking, the ambiguity set $\mc Q$ contains all probability measures under which $\xi$ has a given mean vector $\mu$ and the support of $\Xi$ is contained in the polytope $\Xi$. Using this ambiguity set $\mc Q$, we define the \textit{set of distributionally robust fair allocations}
\begin{align*}
    &\mc X \Let
    \left\{ \begin{array}{l} 
    x \in \mbb R^{|\mc I | \times |\mc J|}: \\
    \Min{\QQ \in \mc Q}\ds \QQ\left( \Big| \sum_{i \in \mc I} \xi_i (x_{ij} - x_{ij'}) \Big| \le \Delta \; \forall (j, j')  \right) \geq 1-\eps
    \end{array}
    \right\}
\end{align*}
that contains allocations $x$ such that the workload among workers are comparatively balanced for \textit{all} probability measures $\QQ$ in $\mc Q$. Instead of solving the stochastic fairness-constrained allocation problem~\eqref{eq:stoprog}, we propose to solve the following \textit{distributionally robust fairness-constrained allocation problem}
\be \label{eq:dro}
        \max \left\{ f(x)  : x \in \mathbb X  \cap \mc X \right\},
\ee
which can also be written in a more explicit form as
\be \notag
	\begin{array}{cl}
		\max & f(x)  \\
		\st & x \in \mathbb X \\
		&     \Min{\QQ \in \mc Q}\ds \QQ\left( \Big| \sum_{i \in \mc I} \xi_i (x_{ij} - x_{ij'}) \Big| \le \Delta \; \forall (j, j') \right) \ge 1-\eps.
	\end{array} 
\ee
Before moving on to Section~\ref{sect:approx} where we provide a solution procedure to problem~\eqref{eq:dro}, we would like to first discuss the components of~\eqref{eq:dro} in further details. While the set of allocations $\mbb X$ has been defined in a minimalistic fashion, we emphasize that there is a great flexibility to incorporate additional operational constraints into $\mbb X$ to promote a higher level of fairness in the assignment. These constraints may include, but are not restricted to, the following:
\begin{itemize}
	\item skillset restrictions: if a worker $j$ does not possess the skill required by task $i$, then the linear constraint $x_{ij} = 0$ can be added into $\mbb X$,
	\item workers' preferences: if a worker $j$ is not willing to perform task $i$, then the linear constraint $x_{ij} = 0$ can be added into $\mbb X$,
	\item pre-assignment: if task $i$ has to assigned to worker $j$, then the linear constraint $x_{ij} = 1$ can be added into $\mbb X$,
	\item knapsack constraints: if worker $j$ can perform at most $\theta_j$ tasks, then the linear constraints $\sum_{i \in \mc I} x_{ij} \le \theta_j$ can be added into $\mbb X$.
\end{itemize}
Similarly, there are various penalty terms that can be incorporated into the loss function $f$ to reflect real-life incurred costs. However, these modifications of the feasible set $\mbb X$ and the objective function $f$ does not alter the reformulation or approximation of the distributionally robust fair allocation set $\mc X$. Because our main goal in this paper is to explore efficient methods to embody the set $\mc X$ into the optimization phase, we thus keep $\mbb X$ and $f$ as simple as possible to avoid any unwanted effects on the quality of the fair allocation.

Returning to the joint chance constraint in~\eqref{eq:dro}, we would like to emphasize that the choice of the \textit{joint} chance constraint is imperatively preferable compared to a set of \textit{individual} chance constraints. Indeed, consider the following set of individual chance constraints
\be \label{eq:individual}
	\left.
	\begin{array}{l}
	\ds \PP\left( \sum_{i \in \mc I} \xi_i (x_{ij} - x_{ij'}) \le \Delta \right) \ge 1-\epsilon_{jj'}^1 \\
	\ds \PP\left( \sum_{i \in \mc I} \xi_i (x_{ij'} - x_{ij}) \le \Delta \right) \ge 1-\epsilon_{jj'}^2
	\end{array}
	\right\} \quad \forall (j, j') \in \mc J \times \mc J
\ee
for a collection of parameters $(\epsilon_{jj'}^1, \epsilon_{jj'}^2)_{(jj') \in \mc J \times \mc J}$. Theoretically, an allocation $x$ which is feasible for the above set of individual chance constraint will satisfy
\[
	\PP\left( \Big| \sum_{i \in \mc I} \xi_i (x_{ij} - x_{ij'}) \Big| \le \Delta \; \forall (j, j') \right) \ge 1 - \epsilon
\]
with $\epsilon = \sum (\epsilon_{jj'}^1 + \epsilon_{jj'}^2)$ by the Bonferroni inequality~\cite[Section~6.1]{ref:prekopa1995stochastic}. This implies that an allocation that satisfies the set of individual chance constraint may not theoretically deliver the necessary guarantee of fairness \textit{jointly} over all pairs of workers with high probability. Consequentially, to achieve a fixed tolerance $\eps$ in the joint chance constraint formulation~\eqref{eq:stoprog}, one may resort to solve an approximation using individual chance constraints of the form~\eqref{eq:individual} with the parameters $\epsilon_{jj'}^1$ and $\epsilon_{jj'}^2$ being set to a sufficiently small value, typically by setting to a common value $\epsilon_{jj'}^1 = \epsilon_{jj'}^2 = \eps/(2 | \mc J |^2)$. Unfortunately, for real-life applications with a large number of workers, this choice of individual parameters tends to be over-conservative and may, in specific cases, fail to even identify a feasible allocation to problem~\eqref{eq:stoprog}. The incapability of the Bonferroni approach to deliver sufficient performance guarantee for the joint chance constraint in the distributionally robust optimization setting has been numerically studied in~\cite{ref:ordoudis2018energy}. This argument justifies the merit of embracing the joint chance constraint formulation in~\eqref{eq:stoprog} and in its distributionally robust counterpart~\eqref{eq:dro}. Finally, one can also augment the threshold $\Delta$ to pair-specific threshold $\Delta_{jj'}$ for each $(j, j')$ to capture more subtle specification of the system. This parameters augmentation will not alter the approach in this paper, thus, we keep a common threshold $\Delta$ for the sake of simplicity.
\section{Tractable Approximations}
\label{sect:approx}

Problem~\eqref{eq:dro} is intuitively appealing, however, it is also notoriously difficult to be solved. Indeed, evaluating whether an allocation $x$ belongs to the set $\mc X$ involves solving an infinite dimensional optimization problem over the space of probability measures. Leveraging on diverse results from the field of distributionally robust optimization, we derive in this section a safe and tractable approximation for problem~\eqref{eq:dro}. To this end, we notice that the argument in the joint chance constraint
\[
	\Big| \sum_{i \in \mc I} \xi_i (x_{ij} - x_{ij'}) \Big| \le \Delta \quad \forall (j, j') \in \mc J \times \mc J
\]
can be written explicitly as a collection of $2 \times | \mc J |$ linear constraints
\[
	\left.
	\begin{array}{r}
	\ds \sum_{i \in \mc I} \xi_i (x_{ij} - x_{ij'})  - \Delta \le 0 \\
	\ds - \sum_{i \in \mc I} \xi_i (x_{ij} - x_{ij'}) - \Delta \le 0 
	\end{array}
	\right\} 
	\quad \forall (j, j') \in \mc J \times \mc J.
\]
Adopting a similar strategy as in~\cite{ref:chen2010from}, we associate the first and second set of constraints with a set of strictly positive scaling factors $\alpha_{jj'}$ and $\beta_{jj'}$ respectively, and the above collection of linear constraints is equivalent to
\[
	\left.
	\begin{array}{r}
	\ds \alpha_{jj'} \Big(\sum_{i \in \mc I} \xi_i (x_{ij} - x_{ij'})  - \Delta \Big) \le 0 \\
	\ds \beta_{jj'} \Big(-\sum_{i \in \mc I} \xi_i (x_{ij} - x_{ij'}) - \Delta \Big) \le 0 
	\end{array}
	\right\} 
	\; \forall (j, j') \in \mc J \times \mc J.
\]
Thus, for any strictly positive scaling matrices $\alpha, \beta \in \R_{++}^{|\mc J| \times |\mc J|}$, the set of probabilistically fair allocations $\mc X$ can be written as
\begin{align*}
    &\mc X = \left\{ \begin{array}{l} 
    x \in \mbb R^{|\mc I | \times |\mc J|}: \\
    \Inf{\QQ \in \mc Q}\ds \QQ\big( \max_{k \in \llbracket 1, K \rrbracket} \{ \xi^\top a_k(x) + b_k \} \le 0 \big) \geq 1-\eps
    \end{array}
    \right\},
\end{align*}
where $K = 2 |\mc J|^2$ is the total number of individual constraints in the joint chance constraint. For any~$k \in \llbracket 1, K \rrbracket$, the $i$-th component of the vector $a_k(x) \in \R^{|\mc I|}$ dependent on $x$ is defined as
\be \label{eq:a-def}
	\begin{aligned}
	&[a_{k}(x)]_{i} = \begin{cases}
		\alpha_{jj'} (x_{ij} - x_{ij'}) & \text{if } k = (j-1) |\mc J| + j', \\
		\beta_{jj'} (x_{ij'} - x_{ij}) & \text{if } k = |\mc J|^2 + (j-1) |\mc J| + j',
	\end{cases}
	\end{aligned}
\ee
for any $i \in \mc I$, and the scalar $b_k \in \R$ is defined as
\be \label{eq:b-def}
	b_{k} = \begin{cases}
		-\alpha_{jj'} \Delta & \text{if } k = (j-1) |\mc J| + j', \\
		-\beta_{jj'} \Delta & \text{if } k = |\mc J|^2 + (j-1) |\mc J| + j'.
	\end{cases}
\ee
We emphasize that the vectors $a_k$ and the scalars $b_k$ are dependent on the specific choice of the scaling factors $\alpha$ and $\beta$, however, this dependence is made implicit at this moment to avoid cluttered notations.
Next, we define the worst-case Conditional Value-at-Risk function
\[
	\mc R_{\alpha, \beta}(x) \Let \Sup{\QQ \in \mc Q} \QQ\text{-CVaR}_\eps \left( \max_{k \in \llbracket 1, K \rrbracket} \{ \xi^\top a_k(x) + b_k \} \right).
\]
Notice that because $a_k(x)$ and $b_k$ are dependent on the scaling matrices $\alpha$ and $\beta$, the worst-case CVaR also depends on these scaling matrices and this dependence is made explicit. We now define the following set
\begin{align*}
    \mc X^\circ(\alpha, \beta) = 
    \left\{ x \in \mbb R^{|\mc I | \times |\mc J|}: \mc R_{\alpha, \beta}(x) \le 0 \right\}.
\end{align*}
We can show that $\mc X^\circ(\alpha, \beta) \subseteq \mc X$ for any strictly positive scaling matrices $\alpha, \beta \in \R_{++}^{|\mc J| \times |\mc J|}$ \cite{ref:nemirovski2006convex}. As a consequence, the optimization problem
\be \label{eq:dro-approx}
    \begin{array}{cl}
        \max & f(x)  \\
        \st & x \in \mathbb X  \cap \mc X^\circ(\alpha, \beta)
    \end{array} 
\ee
constitutes a conservative approximation of the fairness-constrained distributionally robust allocation problem~\eqref{eq:dro}.

\begin{theorem}[Reformulation of worst-case CVaR] \label{thm:refor}
	For any fixed allocation $x \in \R^{|\mc I| \times | \mc J|}$, scaling matrices $\alpha \in \R_{++}^{|\mc J| \times |\mc J|}$ and $\beta \in \R_{++}^{|\mc J| \times |\mc J|}$, the worst-case CVaR $\mc R_{\alpha, \beta}(x)$ equals to the optimal value of the linear optimization problem
\[
	\begin{array}{cl}
            \min & \ds \gamma + \mu^\top \lambda \\ [2ex]
            \st & \gamma \in \R,\, \tau \in \R,\, \lambda \in \R^{|\mc I|}, \, \eta_k \in \R_+^{M} \quad \forall k \in \llbracket 0, K \rrbracket \\
            &\left. 
            \begin{array}{l}
            b_k  - (1 - \eps ) \tau \leq \eps (\gamma - h^\top \eta_k)   \\
            \eps (G^\top \eta_k + \lambda)  = a_k(x)
            \end{array}
            \right\}
            \quad \forall k \in \llbracket 1, K \rrbracket \\
		&\tau \leq \gamma - h^\top \eta_0, \quad G^\top \eta_0 + \lambda = 0.
        \end{array}
\]
\end{theorem}
\begin{proof}
	By denoting $(a)^+ = \max\{0, a\}$, the worst-case CVaR can be expressed as
	\begin{subequations}
	\begin{align}
		&\Sup{\QQ \in \mc Q}~\QQ\text{-CVaR}_\eps \left( \max_{k \in \llbracket 1, K \rrbracket} \{ \xi^\top a_k(x) + b_k \} \right) \notag \\
=& \Sup{\QQ \in \mc Q}\Inf{\tau \in \R} \tau + \frac{1}{\eps} \EE_\QQ\left[\left( \max_{k \in \llbracket 1, K \rrbracket} \{ \xi^\top a_k(x) + b_k\} - \tau\right)^+ \right] \label{eq:cvar1} \\
=& \Sup{\QQ \in \mc Q}\Inf{\tau \in \R}  \EE_\QQ\left[ \max_{k \in \llbracket 0, K \rrbracket} \{ \xi^\top c_k(x) + d_k(\tau)\}  \right] \label{eq:cvar2} \\
=& \Inf{\tau \in \R} \Sup{\QQ \in \mc Q} \EE_\QQ\left[ \max_{k \in \llbracket 0, K \rrbracket} \{ \xi^\top c_k(x) + d_k(\tau)\}  \right], \label{eq:cvar3} 
	\end{align}
	\end{subequations}
	where equality~\eqref{eq:cvar1} exploits the definition of CVaR. In~\eqref{eq:cvar2}, we have defined
\[
	\begin{aligned}
	&c_0(x) = 0,  && d_0(\tau) = \tau , & \\
	&\ds c_k(x) = \frac{a_k(x)}{\eps}, && d_k(\tau) = \frac{b_k}{\eps} + \left(1 - \frac{1}{\eps}\right) \tau \; \forall k \in \llbracket 1, K \rrbracket.
	\end{aligned}
\]
Finally, the equality in~\eqref{eq:cvar3} holds by Sion's minimax theorem~\cite{ref:sion1958minimax} because the objective function is concave in $\QQ$, convex in $\tau$, and the ambiguity set $\mc Q$ is convex and weakly compact thanks to the compactness of the support set $\Xi$. For any fixed value $x$ and for any $\tau$, we can invoke Proposition~\ref{prop:duality} to rewrite the inner supremum problem in~\eqref{eq:cvar3}, we find
\begin{align*}
&\Sup{\QQ \in \mc Q}~\QQ\text{-CVaR}_\eps \left( \max_{k \in \llbracket 1, K \rrbracket} \{ \xi^\top a_k(x) + b_k \} \right) \\
&= \left\{
	\begin{array}{cll}
            \inf & \ds \gamma + \mu^\top \lambda \\ [2ex]
            \st & \gamma \in \R,\, \lambda \in \R^{|\mc I|}, \, \eta_k \in \R_+^{M} & \forall k \in \llbracket 0, K \rrbracket \\
            &\left. 
            \begin{array}{l}
            h^\top \eta_k \leq \gamma - d_k(\tau)  \\
            G^\top \eta_k  = c_k(x) - \lambda 
            \end{array}
            \right\}
            & \forall k \in \llbracket 0, K \rrbracket.
        \end{array}
\right.
\end{align*}
Substituting the expressions for $c_k(x)$ and $d_k(\tau)$ into the above optimization problem and rearranging terms lead to the postulated linear program. 
\end{proof}

From what we have discussed so far, the optimization problem
\be \label{eq:dro-approx-refor}
    \begin{array}{cl}
        \max & f(x) \\
        \st & x \in \mathbb X  \\
		& \gamma \in \R,\, \tau \in \R,\, \lambda \in \R^{|\mc I|}, \, \eta_k \in \R_+^{M} \quad \forall k \in \llbracket 0, K \rrbracket \\
            &\left. 
            \begin{array}{l}
            b_k  - (1 - \eps ) \tau \leq \eps (\gamma - h^\top \eta_k)   \\
            \eps (G^\top \eta_k + \lambda)  = a_k(x)
            \end{array}
            \right\}
            \quad \forall k \in \llbracket 1, K \rrbracket \\
		&\tau \leq \gamma - h^\top \eta_0, \quad G^\top \eta_0 + \lambda = 0 \\
				& \gamma + \mu^\top \lambda \leq 0
    \end{array} 
\ee
is a linear program and it constitutes a tractable conservative approximation of~\eqref{eq:dro} for any \textit{fixed} value $\alpha, \beta \in \R_{++}^{|\mc J| \times |\mc J|}$. More specifically, the optimal value of~\eqref{eq:dro-approx-refor} is a lower bound on the optimal value of~\eqref{eq:dro} and the optimizer $x\opt$ in the variable $x$ of problem~\eqref{eq:dro-approx-refor} is a feasible solution to~\eqref{eq:dro} for any \textit{fixed} value of strictly positive scaling parameters $\alpha, \beta$. One can systematically find the \textit{best} lower bound by jointly optimizing over $\alpha$ and $\beta$ in~\eqref{eq:dro-approx-refor}. Unfortunately, optimizing jointly over the decision variables of~\eqref{eq:dro-approx-refor} and $\alpha, \beta$  is a non-convex optimization problem because it involves the bilinear terms between $x$ and $\alpha, \beta$. Inspired by~\cite{ref:zymler2013distributionally}, we now devise an optimization algorithm that sequentially optimizes over $(\alpha, \beta)$ and the allocation $x$. To this end, we now consider for any \textit{fixed} allocation $x \in \mbb X$ the following optimization problem
\be \label{eq:ab-prob}
		\min \left\{ \mc R_{\alpha, \beta}(x) ~:~ (\alpha, \beta) \in \mc S \right\}
\ee
that searches for $(\alpha, \beta)$ that minimizes the worst-case CVaR of a given allocation $x$. To make problem~\eqref{eq:ab-prob} well-defined, we constrain $(\alpha, \beta)$ to a compact set $\mc S$ defined as
\be \label{eq:S-def}
	\mc S \Let
	\left\{
		\begin{array}{l}
		(\alpha, \beta) \in \R^{|\mc J| \times |\mc J|} \times \R^{|\mc J| \times |\mc J|}: \\
		\ds \sum_{(j, j') \in \mc J \times \mc J}~\alpha_{jj'} = 1,\sum_{(j, j') \in \mc J \times \mc J}~\beta_{jj'} = 1 \\
		\alpha_{jj'} \ge \delta,~\beta_{jj'} \ge \delta \quad \forall (j, j') \in \mc J \times \mc J
		\end{array}
	\right\}
\ee
for some strictly positive but small enough scalar $\delta$ so that $\mc S$ is non-empty. We emphasize that the overall scale of $\alpha$ and $\beta$ is immaterial, and thus normalizing the sum of the elements in $\alpha$ and $\beta$ to 1 in the definition of the set $\mc S$ does not restrict any generality. The generic sequential algorithm is depicted in Algorithm~\ref{algo}.
\begin{algorithm}
	\caption{Generic sequential algorithm}
	\label{algo}
	\begin{algorithmic}
		\REQUIRE Maximum iteration $T$, stopping tolerance $\theta$
		\STATE Initialize $g_0 \leftarrow +\infty$, $t \leftarrow 1$ 
		\STATE Initialize $\alpha_{jj'} \leftarrow 1/ |\mc J|^2$, $\beta_{jj'} \leftarrow 1/ |\mc J|^2 \, \forall (j, j')$
		\WHILE{ $t \le T$}
		\STATE Fix $(\alpha, \beta)$ and find a solution $(x_t, v_t)$ of
		\be \label{eq:x-sub}
			g_t = \left\{
				\begin{array}{cl}
					\max & f(x) -: \mathds{M} v \\
					\st & x \in \mbb X,~v \in \R_+,~\mc R_{\alpha, \beta}(x) \le v.
				\end{array}
			\right.
		\ee
		\STATE \textbf{if} $|g_t - g_{t-1}|/g_t < \theta$ \textbf{then} break \textbf{endif}
		\STATE Find $(\alpha, \beta)$ that solves~\eqref{eq:ab-prob} with $x$ being fixed to $x_t$
		\STATE Set $t \leftarrow t + 1$
		\ENDWHILE
		\ENSURE $x_t$
	\end{algorithmic}
\end{algorithm}

Problem~\eqref{eq:x-sub} can be written explicitly as a linear program in the form of problem~\eqref{eq:dro-approx-refor} two minor modifications: (i)  the right hand side of the last constraint of~\eqref{eq:dro-approx-refor} is replaced by the non-negative decision variable $v$, and (ii) the auxiliary variables is penalized with a {big-$\mathds{M}$} cost parameter in the objective function. This additional decision variable $v$ is necessary: indeed, it relaxes problem~\eqref{eq:dro-approx-refor} and renders problem~\eqref{eq:dro-approx-refor} feasible even when the scaling parameters $(\alpha, \beta)$ are badly initialized.
Moreover, problem~\eqref{eq:ab-prob} can be expressed in the equivalent form as 
\be \label{eq:ab-prob-explicit}
	\begin{array}{cl}
		\min & \gamma + \mu^\top \lambda \\
		\st & (\alpha, \beta) \in \mc S \\
			&\gamma \in \R,\, \tau \in \R,\, \lambda \in \R^{|\mc I|}, \, \eta_k \in \R_+^{M} \quad \forall k \in \llbracket 0, K \rrbracket \\ 
			&\left. 
            \begin{array}{l}
            b_k(\alpha, \beta)  - (1 - \eps ) \tau \leq \eps (\gamma - h^\top \eta_k)   \\
            \eps (G^\top \eta_k + \lambda)  = a_k(\alpha, \beta, x)
            \end{array}
            \right\}
            \quad \forall k \in \llbracket 1, K \rrbracket \\
		&\tau \leq \gamma - h^\top \eta_0, \quad G^\top \eta_0 + \lambda = 0,
	\end{array}
\ee
where $a_k(\alpha, \beta, x)$ and $b_k(\alpha, \beta)$ are defined as in~\eqref{eq:a-def} and~\eqref{eq:b-def}, respectively, with their dependence on $(\alpha, \beta)$ being made explicit. For any allocation $x$, $a_k(\alpha, \beta, x)$ and $b_k(\alpha, \beta)$ are affine functions of $\alpha$ and $\beta$, and hence problem~\eqref{eq:ab-prob-explicit} is a linear program.

The sequence $\{g_t \}_{t \in \mbb N}$ is non-decreasing and thus it is guaranteed to converge. If Algorithm~\ref{algo} terminates with $v_t = 0$, then the allocation $x_t$ is feasible for problem~\eqref{eq:dro} and the terminal value of $g_t$ is a lower bound for the objective value of~\eqref{eq:dro}.

\section{Tractable Approximations for Hyperrectangular Support Set}
\label{sect:hypercube}

Based on the general results of the previous section, we consider now the hyperrectangular support set, where $\Xi$ admits the following specific representation
\be \label{eq:Xi-box}
	\Xi = \left\{ \xi \in \R^{|\mc I|} : l \le \xi \le u \right\}
\ee
for some lower bound vector $l \in \R^{|\mc I|}$ and upper bound vector $u \in \R^{|\mc I|}$ satisfying $l < u$. This hyperrectangular set~\eqref{eq:Xi-box} is useful and practical in many ways. First, prescribing the set $\Xi$ in this case requires the estimation of the upper bound $u_i$ and lower bound $l_i$ for each task $i \in \mc I$, which can be done efficiently given any available dataset. Furthermore, this set $\Xi$ is intuitive, and it is easy to be communicated and explained to the stakeholders involved in the decision making process, including the healthcare workers, the managers of the healthcare unit and the human resources department. Finally, even though the set $\Xi$ may look over-simplified, this choice of the support set has been used widely in the literature and has shown promising performance in inventory management with uncertain demand~\cite{ref:see2010robust}, in vehicle routing problem with uncertain demand or travelling time~\cite{ref:nguyen2016satisficing}. Our aim in this section is to detail the algorithm and the involved optimization problems tailored for the hyperrectangular support set to facilitate the implementation of this model for practical purposes.

When $\Xi$ is a hyperrectangle, the value of the worst-case CVaR~$\mc R_{\alpha, \beta}(x)$ can be reformulated explicitly using the following corollary.
\begin{corollary}[Reformulation for hyperrectangular support set]
	Suppose that the support set $\Xi$ is a hyperrectangle as in~\eqref{eq:Xi-box}. For any fixed allocation $x \in \R^{|\mc I| \times | \mc J|}$, scaling parameters $\alpha \in \R_{++}^{|\mc J| \times |\mc J|}$ and $\beta \in \R_{++}^{|\mc J| \times |\mc J|}$, the worst-case CVaR value $\mc R_{\alpha, \beta}(x)$ equals to the optimal value of the linear optimization problem
\[
	\begin{array}{cl}
            \min & \ds \gamma + \mu^\top \lambda \\ [2ex]
            \st & \gamma \in \R,\, \tau \in \R,\, \lambda \in \R^{|\mc I|}, \\
		& \eta_{1k} \in \R_+^{|\mc I|}, \, \eta_{2k} \in \R_+^{| \mc I|} \quad \forall k \in \llbracket 0, K \rrbracket \\
            &\left. 
            \begin{array}{l}
		b_k  - (1 - \eps ) \tau \leq\eps (\gamma - u^\top \eta_{1k} + l^\top \eta_{2k}) \\
            \eps (\eta_{1k} - \eta_{2k} + \lambda)  = a_k(x)  
            \end{array}
            \right\}
             \forall k \in \llbracket 1,  K \rrbracket \\
		&  \tau \leq \gamma - u^\top \eta_{10} + l^\top \eta_{20}, \quad \eta_{10} - \eta_{20} + \lambda = 0.
        \end{array}
\]
\end{corollary}
This corollary follows directly from Theorem~\ref{thm:refor} by noticing that the support set $\Xi$ in the form~\eqref{eq:Xi-box} can be expressed in the general form~\eqref{eq:Xi-def} with $M = 2 | \mc I|$ constraints using
\[
	G = \begin{bmatrix} 
		I \\
		-I
	\end{bmatrix} \in \R^{2| \mc I| \times | \mc I|},
	\quad
	h = \begin{pmatrix}
		u \\
		-l
	\end{pmatrix} \in \R^{2 | \mc I|}.
\]
Furthermore, when $\Xi$ is dictated by~\eqref{eq:Xi-box}, then for any fixed scaling parameters $\alpha, \beta \in \R_{++}^{| \mc J | \times | \mc J|}$, the optimal value of the program
\be \label{eq:dro-approx-refor-box}
    \begin{array}{cl}
        \max & f(x) \\
        \st & x \in \mathbb X, \, \gamma \in \R,\, \tau \in \R,\, \lambda \in \R^{|\mc I|}, \\
		& \eta_{1k} \in \R_+^{|\mc I|}, \, \eta_{2k} \in \R_+^{| \mc I|} \quad \forall k \in \llbracket 0, K \rrbracket \\
            &\left. 
            \begin{array}{l}
		b_k  - (1 - \eps ) \tau \leq\eps (\gamma - u^\top \eta_{1k} + l^\top \eta_{2k}) \\
            \eps (\eta_{1k} - \eta_{2k} + \lambda)  = a_k(x)  
            \end{array}
            \right\}
             \forall k \in \llbracket 1,  K \rrbracket \\
		&  \tau \leq \gamma - u^\top \eta_{10} + l^\top \eta_{20}, \quad \eta_{10} - \eta_{20} + \lambda = 0 \\
		& \gamma + \mu^\top \lambda \leq 0
    \end{array} 
\ee
provides a valid lower bound on the optimal value of problem~\eqref{eq:dro}. We are now ready to adapt the generic procedure delineated in Algorithm~\ref{algo} to sequentially solve over the decision $x$ and the scaling parameters $\alpha$ and $\beta$ in order to find a competitive solution for problem~\eqref{eq:dro}. Notice that $K = 2|\mc J|^2$, and the optimization problem~\eqref{eq:x-sub} over the decision variables $x$ with $(\alpha, \beta)$ being fixed can be written using the summation notations as
\be \label{eq:fix-alpha-beta}
\begin{array}{cl}
	\max & f(x) - \mathds{M} v \\
	\st & x \in \mathbb X, \, v \in \R_+, \, \gamma \in \R,\, \tau \in \R \\
	& \eta_{1jj'i}^- \in \R_+, \, \eta_{2jj'i}^- \in \R_+ \quad \forall (j, j', i) \in \mc J \times \mc J \times \mc I \\
	& \eta_{1jj'i}^+ \in \R_+, \, \eta_{2jj'i}^+ \in \R_+ \quad \forall (j, j', i) \in \mc J \times \mc J \times \mc I \\
	& \eta_{10i} \in \R_+, \, \eta_{20i} \in \R_+, \, \lambda_i \in \R \quad \forall i \in \mc I\\
	&\left. 
	\begin{array}{l}
		-(1-\eps)\tau - \alpha_{jj'}\Delta \leq \eps \big(\gamma - \ds\sum_{i} (u_i \eta_{1jj'i}^- - l_i \eta_{2jj'i}^-)\big)  \\
		-(1-\eps)\tau - \beta_{jj'}\Delta \leq  \eps \big(\gamma - \ds\sum_{i} (u_i \eta_{1jj'i}^+ - l_i \eta_{2jj'i}^+)\big) 
	\end{array}
	\right\}
	\forall (j,j') \\
	& \left. 
	\begin{array}{l}
		\eps(\eta_{1jj'i}^- - \eta_{2jj'i}^- + \lambda_i) = \alpha_{jj'}(x_{ij} - x_{ij'}) \\
		\eps(\eta_{1jj'i}^+ - \eta_{2jj'i}^+ + \lambda_i) = \beta_{jj'}(x_{ij'} - x_{ij})
	\end{array}
	\right\}
	\forall (j,j',i)\\
	& \tau \leq \gamma - \ds\sum_i(u_i \eta_{10i} - l_i \eta_{20i})\\
	& \eta_{10i} - \eta_{20i} + \lambda_i = 0 \quad \forall i \in \mc I \\
	& \gamma + \ds \sum_i \mu_i \lambda_i \leq v,
\end{array} 
\ee
where the notation $\forall(j, j', i)$ in the constraints means $\forall (j, j', i) \in \mc J \times \mc J \times \mc I$. Problem~\eqref{eq:fix-alpha-beta} is a mixed binary optimization problem because of the binary requirements of the assignment feasible set $\mbb X$. 

Similarly, the optimization problem~\eqref{eq:ab-prob} that solves over $\alpha$ and $\beta$ with the values of $x$ being fixed can be written using the summation notations as
\be \label{eq:fix-x}
	\begin{array}{cl}
	\min & \gamma + \ds \sum_i \mu_i \lambda_i \\
		\st & (\alpha, \beta) \in \mc S,\, \gamma \in \R,\, \tau \in \R \\
	& \eta_{1jj'i}^- \in \R_+, \, \eta_{2jj'i}^- \in \R_+ \quad \forall (j, j', i) \in \mc J \times \mc J \times \mc I \\
	& \eta_{1jj'i}^+ \in \R_+, \, \eta_{2jj'i}^+ \in \R_+ \quad \forall (j, j', i) \in \mc J \times \mc J \times \mc I \\
	& \eta_{10i} \in \R_+, \, \eta_{20i} \in \R_+, \, \lambda_i \in \R \quad \forall i \in \mc I\\
	&\left. 
	\begin{array}{l}
		-(1-\eps)\tau - \alpha_{jj'}\Delta \leq  \eps \big(\gamma - \ds\sum_{i} (u_i \eta_{1jj'i}^- - l_i \eta_{2jj'i}^-)\big) \\
		-(1-\eps)\tau - \beta_{jj'}\Delta \leq   \eps \big(\gamma - \ds\sum_{i} (u_i \eta_{1jj'i}^+ - l_i \eta_{2jj'i}^+)\big) 
	\end{array}
	\right\}
	\forall (j,j') \\
	& \left. 
	\begin{array}{l}
		\eps(\eta_{1jj'i}^- - \eta_{2jj'i}^- + \lambda_i) = \alpha_{jj'}(x_{ij} - x_{ij'}) \\
		\eps(\eta_{1jj'i}^+ - \eta_{2jj'i}^+ + \lambda_i) = \beta_{jj'}(x_{ij'} - x_{ij})
	\end{array}
	\right\}
	\forall (j,j',i)\\
	& \tau \leq \gamma - \ds\sum_i(u_i \eta_{10i} - l_i \eta_{20i})\\
	& \eta_{10i} - \eta_{20i} + \lambda_i = 0 \quad \forall i \in \mc I,
	\end{array}
\ee
where the feasible set $\mc S$ is defined as in~\eqref{eq:S-def}. Furthermore, problem~\eqref{eq:fix-x} is a linear continuous optimization problem, and techniques to warm-start problem~\eqref{eq:fix-x} can be applied to reduce the computational solution time in the iterations of Algorithm~\ref{algo:hyper}.

\begin{algorithm}
	\caption{Specific sequential algorithm for hypercube support set}
	\label{algo:hyper}
	\begin{algorithmic}
		\REQUIRE Maximum iteration $T$, stopping tolerance $\theta$
		\STATE Initialize $g_0 \leftarrow +\infty$, $t \leftarrow 1$ 
		\STATE Initialize $\alpha_{jj'} \leftarrow 1/ |\mc J|^2$, $\beta_{jj'} \leftarrow 1/ |\mc J|^2 \, \forall (j, j')$
		\WHILE{ $t \le T$}
		\STATE Fix $(\alpha, \beta)$ and find a solution $(x_t, v_t)$ of problem~\eqref{eq:fix-alpha-beta}
		\STATE Set $g_t$ to the optimal value of problem~\eqref{eq:fix-alpha-beta}
		\STATE \textbf{if} $|g_t - g_{t-1}|/g_t < \theta$ \textbf{then} break \textbf{endif}
		\STATE Find $(\alpha, \beta)$ that solves~\eqref{eq:fix-x} with $x$ being fixed to $x_t$
		\STATE Set $t \leftarrow t + 1$
		\ENDWHILE
		\ENSURE $x_t$
	\end{algorithmic}
\end{algorithm}
\section{Numerical Experiment}
\label{sect:numerical}

We now showcase the prowess of the distributionally robust optimization framework to promote fair assignments using numerical experiments. To reduce the complexity of the experiments, we concentrate on the hyperrectangular support set as presented in Section~\ref{sect:hypercube}, and find the fair assignment using Algorithm~\ref{algo:hyper}. All codes are implemented in MATLAB and are made available upon request, and the linear optimization problems are solved using Gurobi~9. Due to the nondisclosure agreement, we are unable to report the performance with real-life data, thus we report in this paper only results using synthetic data.

As a benchmark, we propose to compare the solution obtained from Algorithm~\ref{algo:hyper} against the optimal solution of the optimization problem using the expected values as input
\be \label{eq:mean}
\begin{array}{cl}
	\max & f(x) - \mathds{M} v \\
	\st & x \in \mathbb X, \, v \in \R_+ \\
	& \ds -\Delta - v \leq \sum_{i} \mu_i (x_{ij} - x_{ij'}) \leq \Delta + v \quad
	\forall (j,j'),
\end{array} 
\ee
where, once again, the variable $v$ is added to ensure the feasibility of the optimization problem with a penalization parameter big-$\mathds M$ in the objective function. Note that the random quantities $\xi$ are replaced by the mean values $\mu$ in the constraints of problem~\eqref{eq:mean}. 

The numerical experiment is conducted as follows. We set $| \mc I | = 20$ tasks with $| \mc J | = 5$ workers. We generate the vector $\mu \in \R^{20}$ where each element is independent and identically distributed (i.i.d.) with a uniform distribution in $[0, 100]$, and the lower bound $l = \mu - r$ and the upper bound $u = \mu + r$, where $r = \min\{\mu, r'\}$ and $r'$ is a random vector whose elements are i.i.d. with a uniform distribution on $[0, 3]$. The minimum operator in the definition of $r$ is required to make the values of the lower bound $l$ non-negative. The reward $r_{ij}$ is generated i.i.d.~from a uniform distribution on $[0, 100]$ for any $(i, j) \in \mc I \times \mc J$. We set the threshold $\Delta = 5$ and the tolerance $\eps = 5\%$, and then invoke Algorithm~\ref{algo:hyper} to obtain the distributionally robust assignment with $T = 40$ and $\theta = 10^{-4}$, and solve problem~\eqref{eq:mean} to obtain the assignment using the expected values. 

In the test process, we generate 10,000 samples $\{\wh \xi_k\}$ from a uniform distribution on $[l, u]$, where each element is drawn independently of each other. From the $k$-th generated sample $\wh \xi_k$, we record the realized total time that worker $j$ needs to finish all the assigned tasks
\[
	\wh T_j = 	\sum_{i \in \mc I} \wh \xi_{ki} x_{ij},
\]
where $x$ is either the distributionally robust assignment or the assignment using mean values. The spread in the total time, which is a measurement of the degree of fairness in the assignment, is computed as
\be \label{eq:spread}
	s_k = \Max{j} \wh T_j - \Min{j} \wh T_j.
\ee
If this spread exceeds the threshold, that is, $s_k > \Delta$, then the realization $\wh \xi_k$ generates a failure to the task assignment plan. The probability that an assignment fails to satisfy the fairness constraint is computed empirical as the total number of failure scenarios over 10,000 overall samples. 

The aforementioned procedure is replicated 500 times to obtain 500 values of the violation probability for the two approaches. The histogram comparing the probability of violation is presented in Figure~\ref{fig:prob}. One can observe that the distributionally robust optimization approach successfully promotes the fairness in the assignment by significantly reducing the probability that the total time spread is bigger than the threshold $\Delta$. Averaging over 500 replications, the probability of violation is 1.37\% for the distributionally robust optimization and is 61.73\% for the optimization using the expected values approach. Notice that the average probability of 1.37\% is significantly lower than the prescribed tolerance of $\eps = 5\%$. This effect stems from the combination of the distributionally robust approach and the conservativeness of the joint chance constraint approximation.

Figure~\ref{fig:spread} compares the spreads of one particular replication using 10,000 samples. It is obvious that the spreads is smaller for the distributionally robust approach compared to the optimization using the expected values approach. Furthermore, from Figure~\ref{fig:spread-a}, the empirical spread is kept below the threshold $\Delta = 5$ with high probability, which demonstrates the effectiveness of the distributionally robust optimization approach in satisfying the joint chance constraint.

We emphasize that guaranteeing the fairness in the task assignment does not come for free. Indeed, a fair assignment which is obtained by solving the distributionally robust optimization problem often results in a lower reward compared to the optimization with the expected values approach. Figure~\ref{fig:reward} depicts the rewards from each approach taken over 500 replications. The mean reward for the distributionally robust optimization approach is 527.62, and this quantity for the optimization using the expected values is 626.90.

\begin{figure}
	\centering
	\subfloat[Distributionally robust optimization approach.]{%
		\resizebox*{6.5cm}{!}{\includegraphics{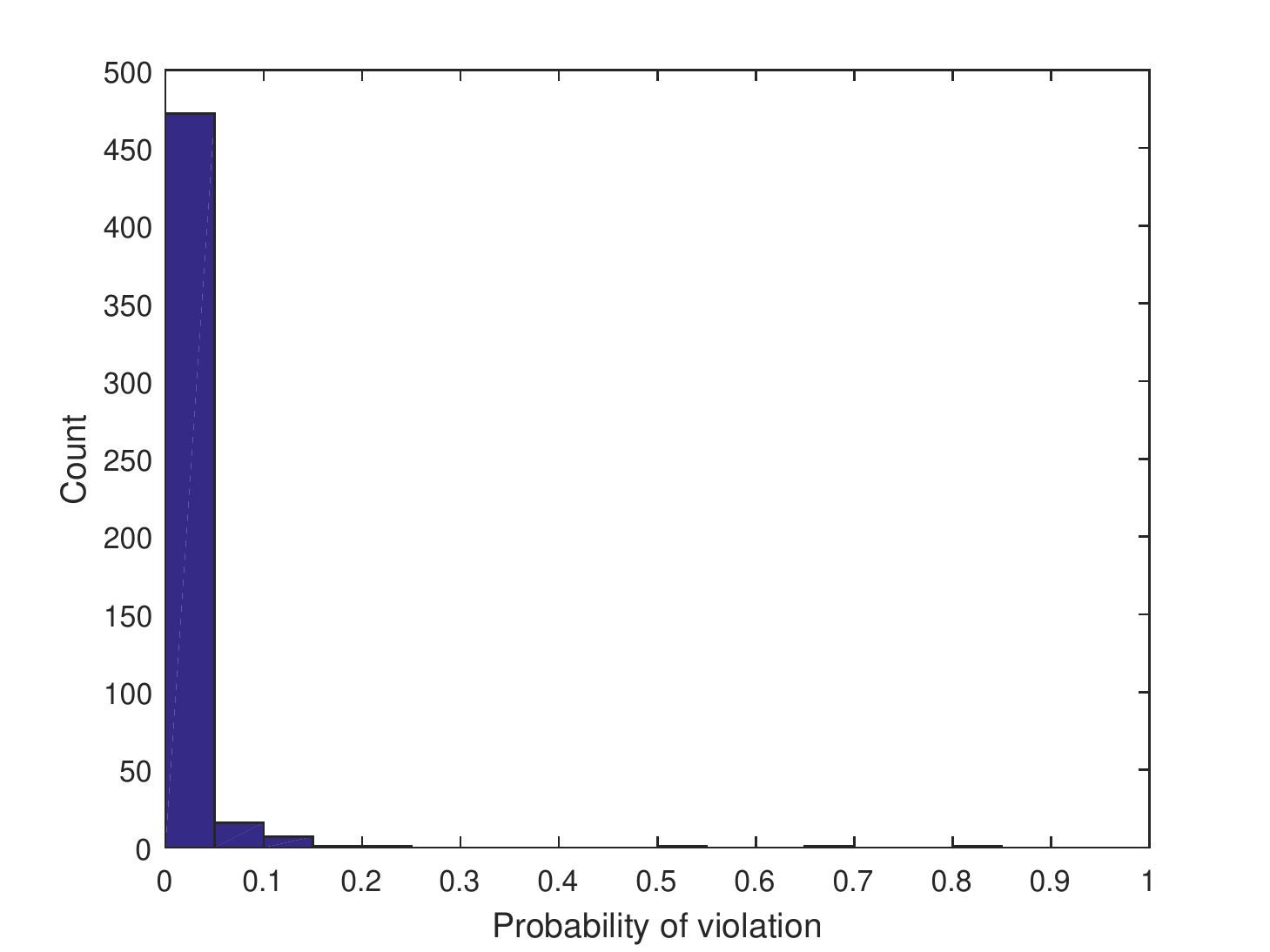}}}\hspace{5pt}
	\subfloat[Optimization with expected values approach.]{%
		\resizebox*{6.5cm}{!}{\includegraphics{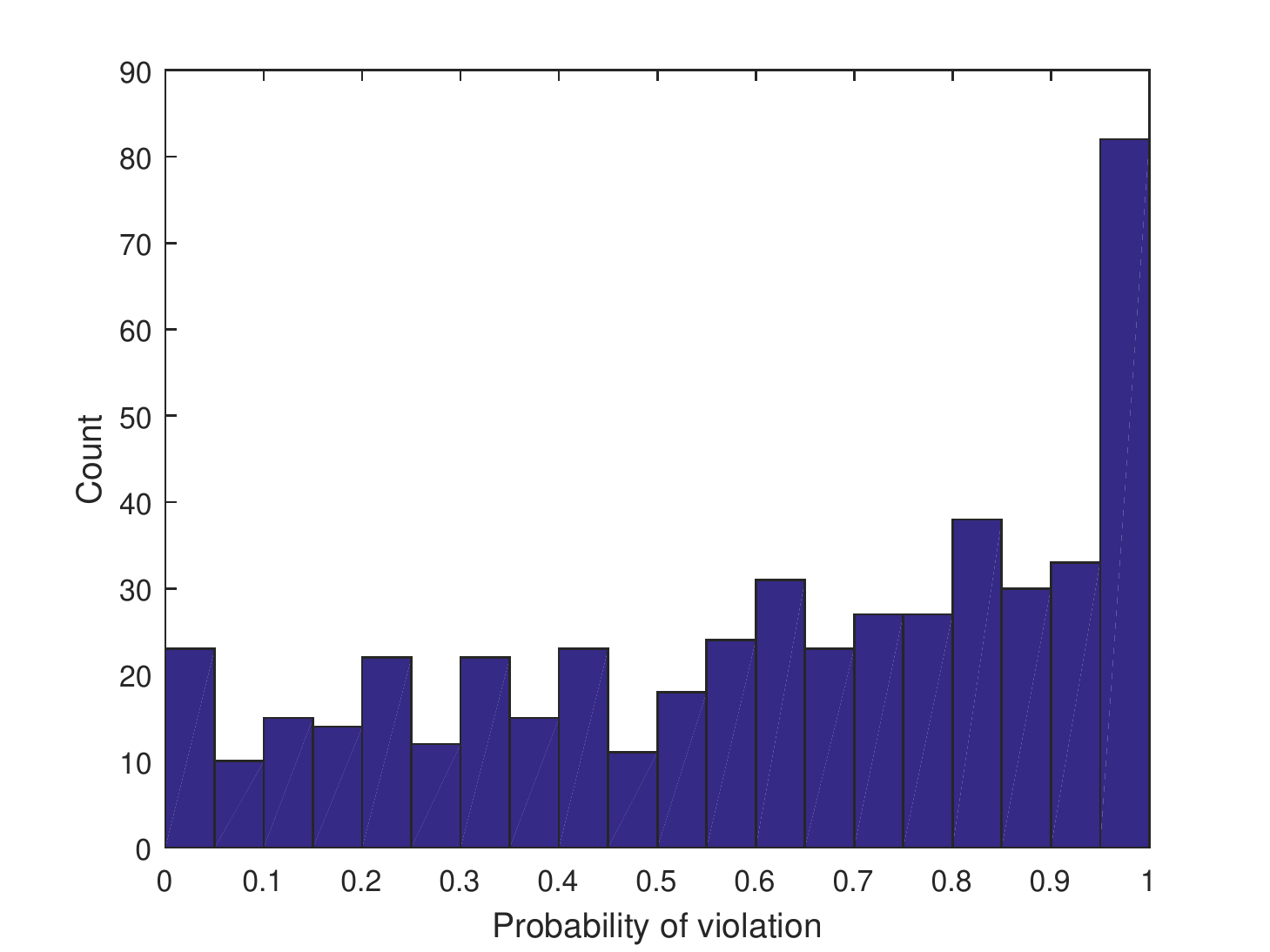}}}
	\caption{Probability of violation taken over 500 replications for both approaches.}
	\label{fig:prob}
\end{figure}

\begin{figure}
	\centering
	\subfloat[Distributionally robust optimization approach.]{%
		\resizebox*{6.5cm}{!}{\includegraphics{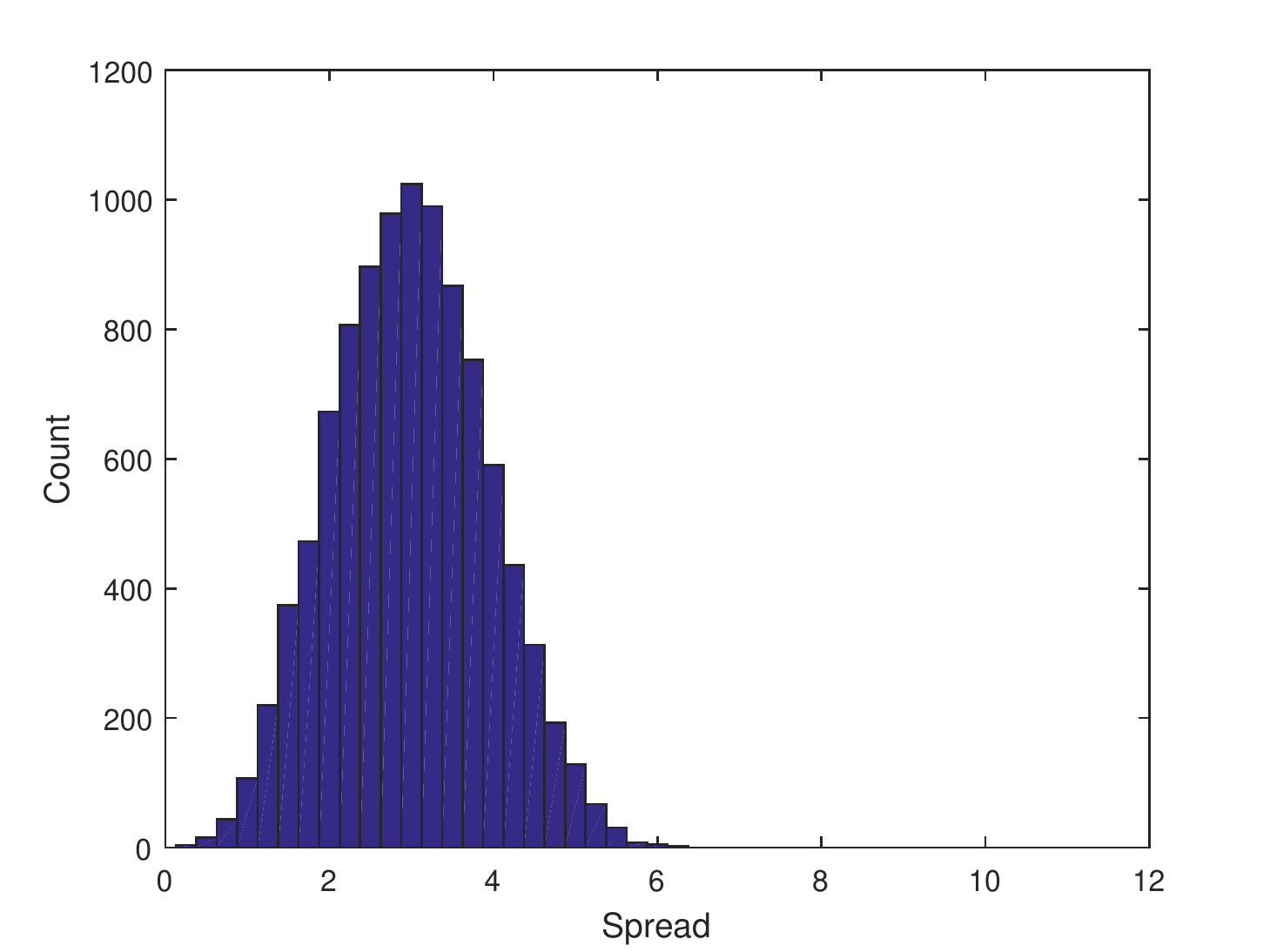} \label{fig:spread-a}}}\hspace{5pt}
	\subfloat[Optimization with expected values approach.]{%
		\resizebox*{6.5cm}{!}{\includegraphics{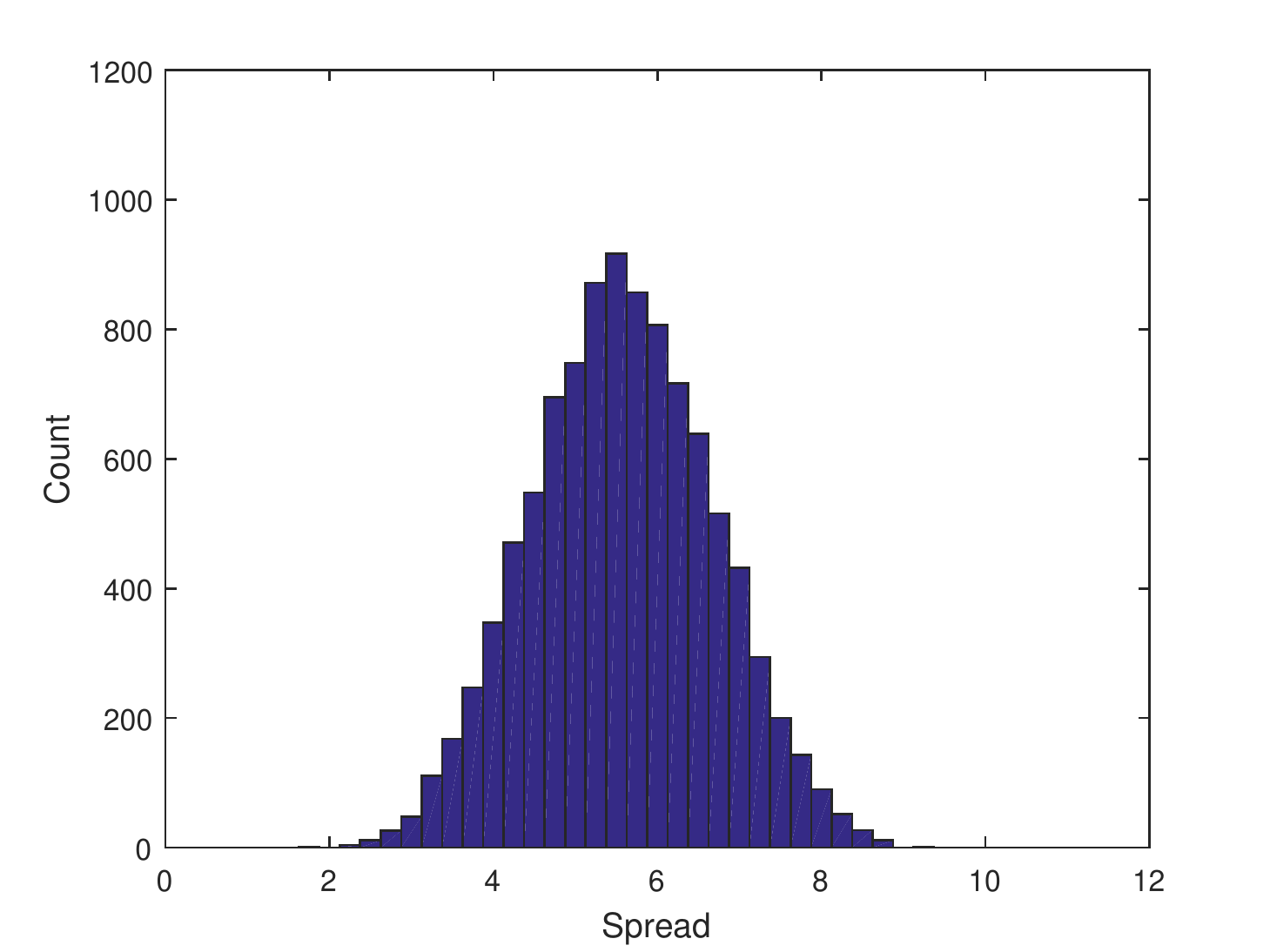}}}
	\caption{Spreads, calculated in unit time as in~\eqref{eq:spread} of a particular replication with 10,000 samples.}
	\label{fig:spread}
\end{figure}

\begin{figure}
	\centering
	\subfloat[Distributionally robust optimization approach.]{%
		\resizebox*{6.5cm}{!}{\includegraphics{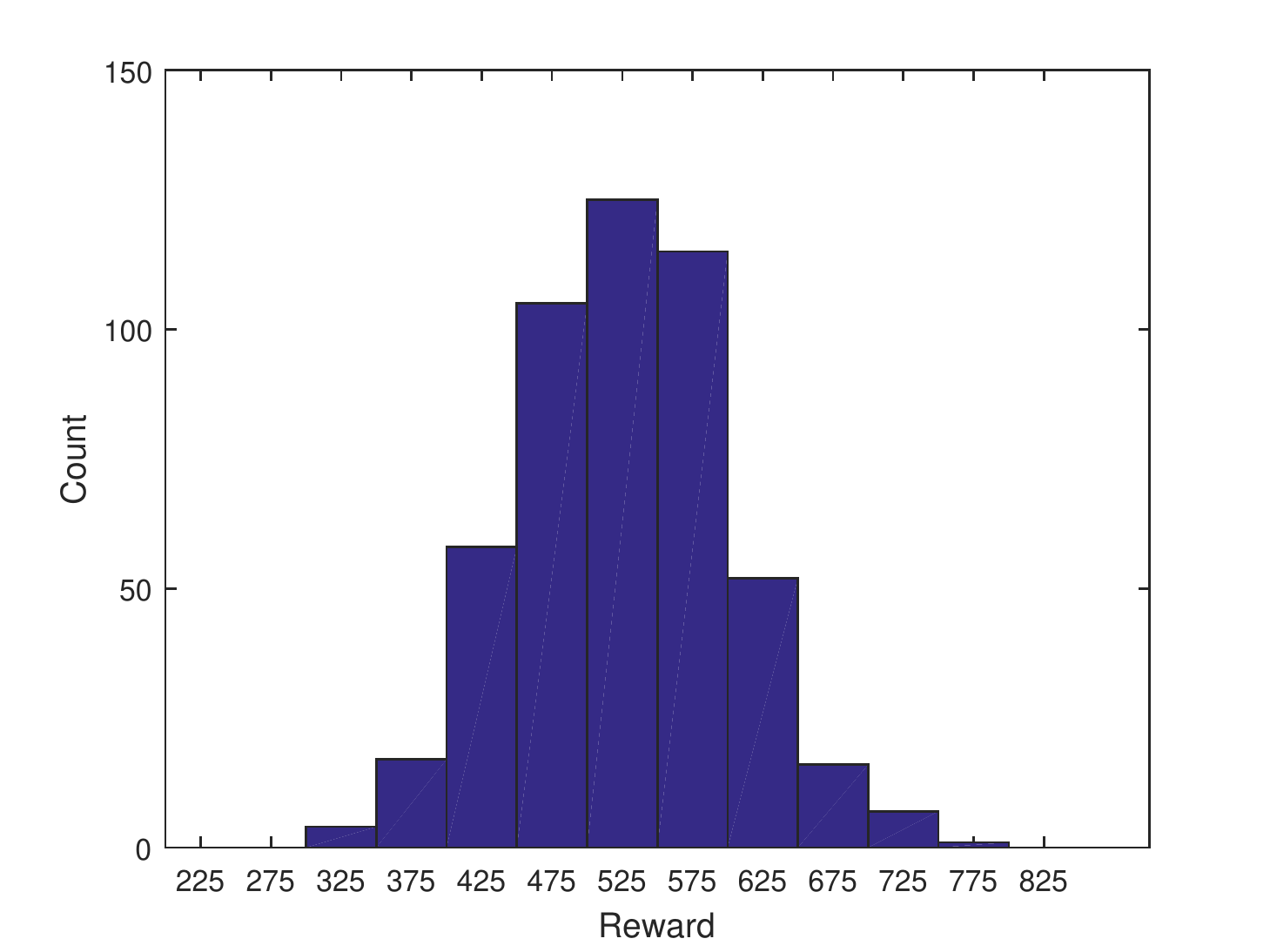} \label{fig:reward-a}}}\hspace{5pt}
	\subfloat[Optimization with expected values approach.]{%
		\resizebox*{6.5cm}{!}{\includegraphics{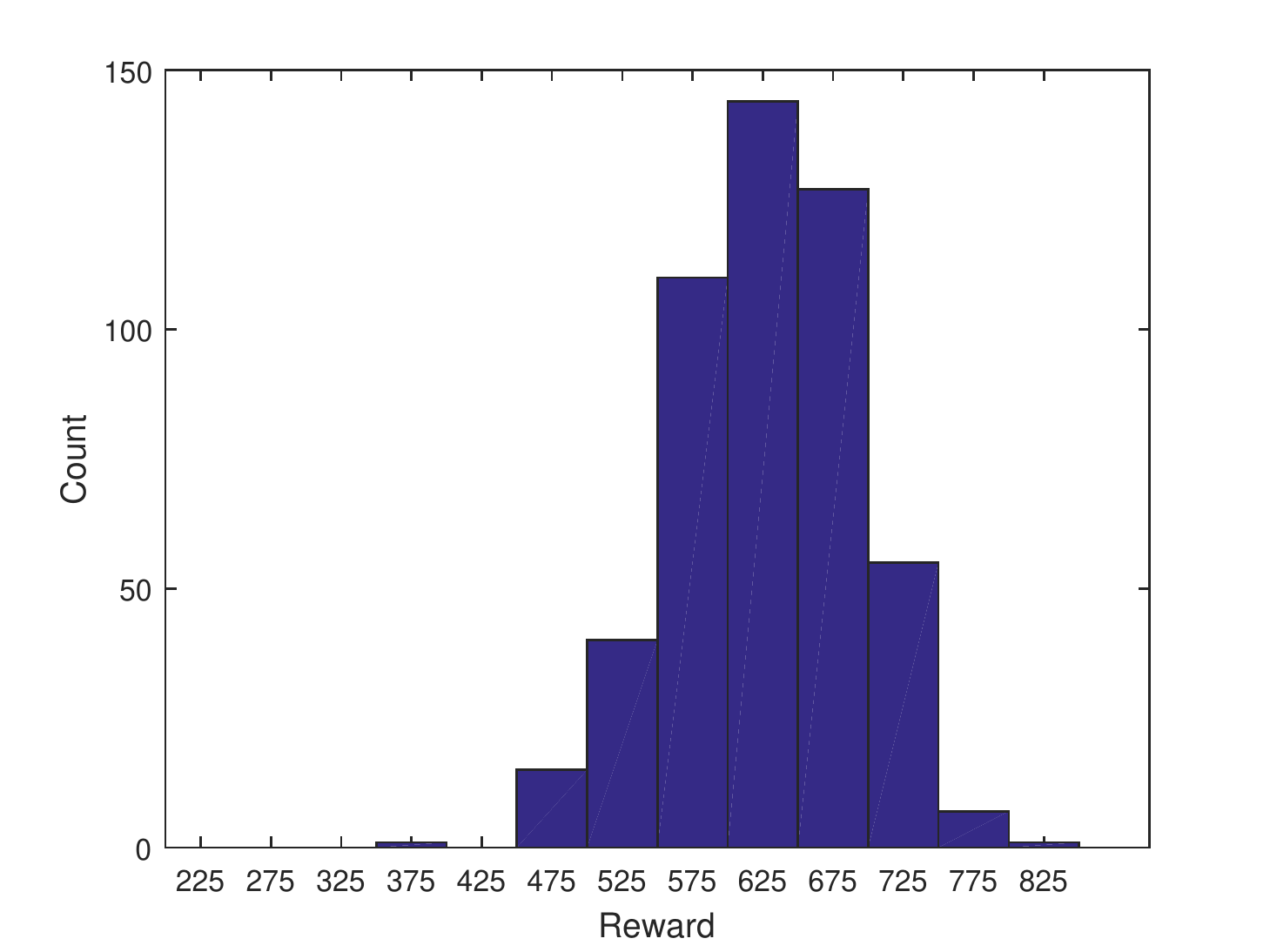}}}
	\caption{Rewards taken over 500 replications for both approaches.}
	\label{fig:reward}
\end{figure}

\section{Concluding Remarks}
\label{sect:conclusions}

In this paper, we explore how the methodology of distributionally robust optimization can be applied to generate fair task assignment. The notion of fairness is measured by the difference in the time the workers are required to finish the assigned tasks. We induce a fair assignment by imposing a joint chance constraint that bounds the time difference by a certain threshold with high probability, and we formulate the distributionally robust optimization problem with joint chance constraint to alleviate the depedence on the identification of the true underlying distribution, which is sometimes even impossible. We develop a conservative approximation of the joint chance constraint using Conditional Value-at-Risk, and propose a sequential optimization algorithm to efficiently solve the resulting reformulations. Finally, we tailor the result to a hyperrectangular support set and report promising numerical results using synthetic datasets. The results in this paper are a substantial part of the efforts from the hospital management team of the Hanoi Obstetrics and Gynecology Hospital to improve the job satisfaction of our healthcare workers, which in consequence aims to improve the quality of our healthcare service. 

The results in this paper reveal several directions for future research. While this paper focuses on the mean-support ambiguity set, many other types of ambiguity sets are available and can be used to promote fairness in the allocation of tasks to healthcare workers. These new ambiguity sets include, but is not restricted to, the ambiguity sets prescribed using the $\phi$-divergence~\cite{ref:ben2015deriving} and the Wasserstein distance~\cite{ref:kuhn2019wasserstein}. Moreover, as we have noted in Section~\ref{sect:numerical}, the assignment using the distributionally robust optimization approach usually have lower rewards compared to other approaches which is less stringent on the fairness criterion. However, this reduction in the reward can be relieved in a certain extent by injecting more flexibility in the recourse actions to adapt to the realization of the uncertain quantity as the plan rolls over. This approach can be implemented using a re-optimization on the second-stage decisions, which has been first demonstrated in the power systems scheduling setting~\cite{ref:ordoudis2018energy} and can be potentially applied to this problem. We leave these ideas for future research.

\section*{Proofs}
\label{sect:appendix}
\begin{proposition}[Strong duality] \label{prop:duality}
    Suppose that the ambiguity set $\mc Q$ is defined as in~\eqref{eq:ambi-def} and that the loss function $\ell(\xi)$ is a convex, piecewise affine function of the form $\ell(\xi) = \max_{k \in \llbracket 0, K \rrbracket} \{ c_k^\top \xi + d_k \}$ for some $c_k \in \R^{|\mc I|}$ and $d_k \in \R~\forall k \in \llbracket 0, K \rrbracket$. The worst-case expected loss can be reformulated as
    \begin{align*}
        &\Max{\QQ \in \mc Q}~\EE_{\QQ}[\ell(\xi)] = \left\{
            \begin{array}{cll}
            \min & \ds \gamma + \mu^\top \lambda \\ [2ex]
            \st & \gamma \in \R,\, \lambda \in \R^{|\mc I|}, \, \eta_k \in \R_+^{M} & \forall k \in \llbracket 0, K \rrbracket \\
            &\left. 
            \begin{array}{l}
            h^\top \eta_k \leq \gamma - d_k  \\
            G^\top \eta_k  = c_k - \lambda 
            \end{array}
            \right\}
            & \forall k \in \llbracket 0, K \rrbracket.
        \end{array}
        \right.
    \end{align*}
\end{proposition}

The result presented in Proposition~\ref{prop:duality} can be recovered as a special case of~\cite[Theorem~1]{ref:wiesemann2014distributionally}. An elementary proof is presented here for completeness.

\begin{proof} By exploiting the definition of the ambiguity set $\mc Q$, the worst-case expected loss can be written as an infinite dimensional optimization problem
\be \label{eq:worst-case-expection}
        \Max{\QQ \in \mc Q}~\EE_{\QQ}[\ell(\xi)] = 
        \left\{
            \begin{array}{cl}
                \max &  \ds \int_{\R^{|\mc I|}} \ell(\xi)~\QQ(\mathrm{d} \xi) \\
                \st & \QQ \in \mc M(\R^{|\mc I|})\\
                &\ds \int_{\R^{|\mc I|}}\mathbbm{1}_{\Xi}(\xi) ~\QQ(\mathrm{d} \xi) = 1 \\
                & \ds \int_{\R^{|\mc I|}} \xi~\QQ(\mathrm{d} \xi) = \mu,
            \end{array}
        \right.
    \ee
    where $\mc M(\R^{|\mc I|})$ denotes the set of all positive measures  on $\R^{|\mc I|}$ and $\mathbbm{1}_{\Xi}$ is the indicator of the set $\Xi$, that is,
    \[
        \mathbbm{1}_{\Xi}(\xi) = \begin{cases}
            1 & \text{if } \xi \in \Xi, \\
            0 & \text{otherwise}.
        \end{cases}
    \]
    Because $\mu$ lies in the interior of the support set $\Xi$ by assumption, the Slater type condition holds for problem~\eqref{eq:worst-case-expection}. As a consequence, strong duality holds by~\cite[Proposition~3.4]{ref:shapiro2001on}, and the optimal value of problem~\eqref{eq:worst-case-expection} equals to the optimal value of the following semi-infinite optimization problem
    \be \label{eq:dual}
        \begin{array}{cl}
            \min & \ds \gamma + \mu^\top \lambda \\ [2ex]
            \st & \lambda \in \R^{|\mc I|}, \, \gamma \in \R \\
            & \ds \gamma + \xi^\top \lambda \ge \ell(\xi) \quad \forall \xi \in \Xi.
        \end{array}
    \ee
    By exploiting the piecewise affine form of the loss function $\ell$, the constraint of problem~\eqref{eq:dual} can be re-expressed as a system of $K$ semi-infinite linear constraints
    \[
        \ds \gamma + \xi^\top \lambda \ge \xi^\top c_k  + d_k \quad \forall \xi \in \Xi\quad \forall k \in \llbracket 0, K \rrbracket,
    \]
    which is further equivalent to a system of $K$ robust linear constraints
    \[
        \ds \max_{\xi \in \Xi}~\xi^\top (c_{k} - \lambda) \le \gamma -  d_k \quad  \forall k \in \llbracket 0, K \rrbracket
    \]
    Formulating the dual linear program of the supremum problem on the left hand side of the above constraint, we find that problem~\eqref{eq:dual} is equivalent to
    \[
        \begin{array}{cll}
            \min & \ds \gamma + \mu^\top \lambda \\ [2ex]
            \st & \gamma \in \R,\, \lambda \in \R^{|\mc I|}, \, \eta_k \in \R_+^{M} & \forall k \in \llbracket 0, K \rrbracket \\
            &\left. 
            \begin{array}{l}
            h^\top \eta_k \leq \gamma - d_k  \\
            G^\top \eta_k  = c_k - \lambda 
            \end{array}
            \right\}
            & \forall k \in \llbracket 0, K \rrbracket.
        \end{array}
    \]
    This observation completes the proof. 
\end{proof}

\bibliographystyle{siam}
\bibliography{Fair.bbl}

\end{document}

%% file: style.tex
\usepackage{float,graphicx}
\usepackage{amssymb,amsmath, color}
\usepackage{algorithmic, algorithm}

\usepackage{amsthm}
\usepackage{dsfont, bbm, stmaryrd}
\usepackage{hyperref}

%% file: comments.tex
\makeatletter
\def\munderbar#1{\underline{\sbox\tw@{$#1$}\dp\tw@\z@\box\tw@}}
\makeatother

\newtheorem{theorem}{Theorem}[section]

\newtheorem{corollary}[theorem]{Corollary}
\newtheorem{proposition}[theorem]{Proposition}

\newcommand{\be}{\begin{equation}}
\newcommand{\ee}{\end{equation}}
\newcommand{\bea}{\begin{equation*}\begin{aligned}}
\newcommand{\eea}{\end{aligned}\end{equation*}}
\newcommand{\ds}{\displaystyle}

\newcommand{\R}{\mathbb{R}}

\newcommand{\Max}{\max\limits_}
\newcommand{\Min}{\min\limits_}
\newcommand{\Sup}{\sup\limits_}
\newcommand{\Inf}{\inf\limits_}
\newcommand{\st}{\mathrm{s.t.}}

\newcommand{\wh}{\widehat}
\newcommand{\mc}{\mathcal}
\newcommand{\mbb}{\mathbb}

\newcommand{\PP}{\mbb P}

\newcommand{\QQ}{\mbb Q}

\newcommand{\Let}{\triangleq}
\newcommand{\opt}{^\star}
\newcommand{\eps}{\varepsilon}

\newcommand{\EE}{\mathds{E}}

\newcommand{\m}{\mu}